\documentclass[11pt,a4paper,headinclude,footinclude,fleqn]{article}                 
\usepackage[T1]{fontenc}                   
\usepackage[utf8]{inputenc}                 
\usepackage[english]{babel}       
\usepackage{graphicx}                       %
\usepackage[font=small]{quoting}            %
\usepackage{caption}                        %
\usepackage[nochapters,beramono,eulermath,  %
            pdfspacing,
            listings,
            ]{classicthesis}                %
\usepackage{arsclassica}                    
\usepackage[top=.8in,bottom=1.2in,left=1.2in,right=1.2in]{geometry}
\usepackage[english]{babel}
\usepackage{verbatim}
\usepackage{color}
\usepackage{picture}
\usepackage{amsmath,amsthm,amsfonts,amssymb}
\usepackage{comment}
\usepackage{mathtools,bigints}

\usepackage{titling} 
\usepackage[hyperpageref]{backref}

\newtheorem{thm}{Theorem}[section]
\newtheorem{lem}[thm]{Lemma}

\newtheorem{prop}[thm]{Proposition}

\theoremstyle{definition}
\newtheorem{rem}[thm]{Remark}
\theoremstyle{remark}

\newcommand{\ds}{\displaystyle}

\newcommand{\R}{\mathbb{R}}
\newcommand{\N}{\mathbb{N}}

\newcommand{\de}{\partial}
\newcommand{\eps}{\varepsilon}

{\left\{\begin{array}{@{}l@{}}}{\end{array}\right.}
\patchcmd{\abstract}{\scshape\abstractname}{\textbf{\abstractname}}{}{}
\makeatletter 
\def\@makefnmark{} 
\makeatother 
 
\usepackage{authblk}

\title{Saturation phenomena for some classes of nonlinear nonlocal eigenvalue problems}
\author[,a]{Francesco Della Pietra\thanks{f.dellapietra@unina.it}}
\author[,b]{Gianpaolo Piscitelli\thanks{gianpaolo.piscitelli@unicas.it}}
\affil[a]{\scriptsize Dipartimento di Matematica e Applicazioni ``R. Caccioppoli'', Universit\`a degli studi di Napoli Federico II \\ Complesso Universitario Monte S. Angelo, Via Cintia 45, 80126 Napoli, Italy.}
\affil[b]{\scriptsize Dipartimento di Ingegneria Elettrica e dell'Informazione, Universit\`a degli Studi di Cassino e del Lazio Meridionale\\ Via G. Di Biasio n. 43, 03043 Cassino (FR), Italy.}

\begin{document}
\maketitle
\begin{abstract}
Let us consider the following minimum problem 
\[
\lambda_\alpha(p,r)=\min_{\substack{u\in W_{0}^{1,p}(-1,1)\\u\not\equiv0}}\dfrac{\ds\int_{-1}^{1}|u'|^{p}dx+\alpha\left|\int_{-1}^{1}|u|^{r-1}u\, dx\right|^{\frac pr}}{\ds\int_{-1}^{1}|u|^{p}dx},
\]
where $\alpha\in\R$, $p\ge 2$ and $\frac p2 \le r \le p$. We show that there exists a critical value $\alpha_C=\alpha_C (p,r)$ such that the minimizers have constant sign up to $\alpha=\alpha_{C}$ and then they are odd when $\alpha>\alpha_{C}$.
\end{abstract}

\section{Introduction}
In this paper we consider the problem:
\begin{equation}\label{operat}
\lambda_\alpha(p,r)=\inf \left\{ \mathcal Q_\alpha[u],\; u\in W_0^{1,p}(-1,1),\,u\not\equiv 0 \right\},
\end{equation}
where 
\begin{equation}
\label{Qualpha}
\mathcal Q_{\alpha}[u]:=\dfrac{\ds\int_{-1}^{1}|u'|^{p}dx+\alpha\left|\int_{-1}^{1}|u|^{r-1}u\, dx\right|^{\frac pr}}{\ds\int_{-1}^{1}|u|^{p}dx},
\end{equation}
with $\alpha\in\R$ and $1\le \frac{p}{2}\le r \le p$.

The problem we deal with has been treated by many authors both in the one dimensional and in the $n$-dimensional case. For example, reaction-diffusion equations describing chemical processes (see \cite{F}, \cite{S}) or Brownian motion with random jumps (see \cite{P}).

The minimization problem \eqref{operat} leads, in general, to a nonlinear eigenvalue problem with a nonlocal term. 
Supposing without loss of generality that $y$ is a minimizer with $\int_{-1}^1 |y|^{r-1}y\  dx\ge 0$, we have
\begin{equation*}
\left\{
\begin{array}{ll}
-(|y'|^{p-2}y')' + \alpha\left(\ds\int_{-1}^1 |y|^{r-1}y\  dx\right)^{\frac{p}{r}-1}|y|^{r-1}=\lambda_\alpha(p,r)\, |y|^{p-2}y\quad\text{in}\ ]-1,1[\\[.4cm]
y(-1)=y(1)=0
\end{array}
\right.
\end{equation*}
(see Section 2 for its precise statement). 

The value $\lambda_\alpha(p,r)$ is the optimal constant in the Sobolev-Poincar\'e-Wirtinger inequality
\begin{equation*}
\lambda_\alpha(p,r)\int_{-1}^{1}|u|^{p}dx\le{\ds\int_{-1}^{1}|u'|^{p}dx+\alpha\left|\int_{-1}^{1}|u|^{r-1}u\, dx\right|^{\frac pr}}{\ds},
\end{equation*}
which holds for any $u\in W_0^{1,p}(-1,1)$. Our aim is to study symmetry properties of the minimizers of \eqref{operat} and, as a consequence, to give some informations on $\lambda_\alpha(p,r)$. In the local case ($\alpha=0$), this inequality reduces to the classical one-dimensional Poincar\'e inequality; in particular, 
\[
\lambda_{0}(p,r)=\left(\frac{\pi_p}{2}\right)^p
\]
for any $p$ and $r$, where
\[
\pi_p=2\int_0^{+\infty}\frac{1}{1+\frac{1}{p-1}s^p}ds=2\pi\frac{(p-1)^{\frac1p}}{p\sin\frac{\pi}{p}}.
\]

Our problem is related to the study of the minimization of \eqref{Qualpha} under the assumption $\int_{-1}^{1}|u|^{r-1}u=0$ (that is the limit case ``$\alpha=\infty$''). This was studied by several authors (see for example \cite{DGS,E,BKN,BK,N1,CD,GN}), considering various cases of the exponents $p,q,r$. A very general case was studied recently in \cite{GGR}, where the authors studied the symmetry of the minimizers of
\begin{equation}
\label{twist}
\tilde \Lambda(p,q,r):=\min \left\{ \dfrac{\ds\int_{-1}^{1}|u'|^{p}dx}{\ds\left(\int_{-1}^{1}|u|^{q}dx\right)^\frac pq},\; u\in W_0^{1,p}(-1,1),\,\int_{-1}^{1}|u|^{r-1}u\,dx=0,\,u\not\equiv 0 \right\},
\end{equation}
and showed that, when $p,q>1$, $r>0$ with $q\le (2r+1)p$, these minimizers are odd functions. In particular, if $1<p=q<\infty$, they showed that
\begin{equation*}
\Lambda(p,r):=\tilde \Lambda(p,p,r)=\pi_p^p
\end{equation*}
for any $r$. In \cite{DP} we studied the problem \eqref{operat} in the case $p=2$. In this paper we consider the more general case $p\ge 2$. Recently, this problem was studied also in the multidimensional case, when $\alpha \in \R$ and $p=q=2$ in \cite{BFNT} ($r=1$) and in \cite{D} ($r=2$). 
For related problems we refer the reader to
\cite{FH,Narxiv,BDNT,BCGM,CHP,KN,Pi,BCGM}.

In the present paper, we show that the nonlocal term affects the minimizer of problem \eqref{operat} in the sense that it has constant sign up to a critical value of $\alpha$ and, for $\alpha$ larger than the critical value, it has to change sign, and a saturation effect occurs. More precisely, the first main result we obtain is the following.
\begin{thm}
\label{mainthm1}
Let $p\geq2$, $\frac p2\le r \le p$. Then there exists a positive number $\alpha_C=\alpha_C(p,r)$ such that:
\begin{enumerate}
\item if $\alpha<\alpha_{C}$, then
\[
\lambda_\alpha(p,r)<\pi_p^{p},
\]
and any minimizer $y$ of $\lambda_\alpha(p,r)$ has constant sign in $]-1,1[$.
\item If $\alpha\ge\alpha_{C}$, then
\[
\lambda_\alpha(p,r)= \pi_p^{p}.
\]
Moreover, if $\alpha>\alpha_{C}$, the function {$y(x)=\sin_p\pi_p x$, $x\in[-1,1]$, is the unique minimizer, up to a multiplicative constant, of $\lambda_\alpha(p,r)$. Hence it is odd, $\int_{-1}^{1} |y(x)|^{r-1}y(x)\,dx=0$, and $\overline x=0$ is the only point in $]-1,1[$ such that $y(\overline x)=0$.
}
\end{enumerate}
\end{thm}
Moreover we analyze the behaviour of the minimizers associated to the critical values.
\begin{thm}
\label{mainthm2}
Let $p\geq 2$, $\frac p2\le r \le p$,
if $\alpha=\alpha_C(p,r)$, then $\lambda_{\alpha_C}(p,r)$ admits both a positive minimizer and the minimizer $y(x)=\sin_p\pi_p x$, up to a multiplicative constant. Moreover, if $r>\frac p2$ any minimizer has constant sign or it is odd. Furthermore, if $r=p$, then $\alpha_{C}(p,p)=\frac{2^p-1}{2^p}\pi_p^{p}$.
\end{thm}
\begin{rem}
When the interval is $]a,b[$ instead of $]-1,1[$, we have
\[
\lambda_{\alpha}(p,r;]a,b[)= \left(\frac{2}{b-a}\right)^{p} \cdot 
\lambda_{\tilde \alpha}\left(p,r\right),
\]
\end{rem}
with $\tilde \alpha=\left(\frac{b-a}{2}\right)^{\frac pr +p-1}\alpha$.
The outline of the paper follows. In Section 2 we show some properties of $\lambda_\alpha(p,r)$, while in Section 3 we study the behavior of the changing-sign minimizers. Finally, in Section 4 we give the proof of the main results.

\section{Preliminaries}
\subsection{The p-circular functions} Let $p>1$ and let us consider the function $F_p:[0, (p-1)^\frac 1p]$ defined as
\begin{equation*}
F_p(x)=\ds\int_0^x\frac{dt}{\ds\left[1-{t^p}/{(p-1)}\right]^\frac 1p}.
\end{equation*}
Denote by $z(s)$ the inverse function of $F$ which is defined on the interval $\left[0,\frac{\pi_p}{2}\right]$, where
\[
\pi_p=2\ds\int_{ 0}^{(p-1)^\frac 1p}\frac{dt}{\ds\left[1-{t^p}/{(p-1})\right]^\frac 1p}=2(p-1)^\frac1p \int_0^1\frac{dx}{\ds(1-x^p)^\frac1p}.\]
We define $\sin_p$, the $p$-sine function, as the following periodic extension of $z(t)$:
\[
\sin_p(t)=\left\{ \begin{split}
& z(t) &&\text{if}\ \ t\in\left[0,\frac{\pi_p}{2}\right],\\
& z(\pi_p-t) &&\text{if}\ \ t\in\left[\frac{\pi_p}{2}, \pi_p\right], \\
& -\sin_p(-t)\  &&\text{if}\ \ t\in\left[-\pi_p, 0\right]. \\
\end{split}
\right.
\]
It is extended periodically to all $\R$, with period $2\pi_p$. The $p$-cosine function is defined as
\[
\cos_p( t) =\sin_p \left( t + \frac{\pi_p}{2}\right)
\]
and it is again an even function with period $2 \pi_p$. Let us explicitely observe that these generalized sine and cosine function coincide with the usual ones when $p=2$ and that they have continuous second derivative if $1<p<2$ and continuous first derivative if $2<p<\infty$ (see \cite{O}). For further details we refer for example to \cite{L}. 
The study of the $p$-circular functions is connected with the $1$-dimensional Dirichlet $p$-Laplacian eigenvalue problem. Indeed, the minimum $\lambda_p$ of the Rayleigh quotient
\[
\mathcal Q_p(u)=\frac{\ds\int_{-1}^1 |u'(x)|^p \ dx}{\ds\int_{-1}^1 |u(x)|^p \ dx} \qquad (1<p<\infty),
\]
 among all real valued functions $u\in W_{0}^{1,p}$, is the first eigenvalue $\lambda_p$ of the problem
\begin{equation*}
\left\{
\begin{array}{ll}
-(|y'|^{p-2}y')' =\lambda_p\, |y|^{p-2}y& \text{in}\ ]-1,1[\\[.4cm]
y(-1)=y(1)=0.
\end{array}
\right.
\end{equation*}
This first eigenvalue is just $\left(\frac{\pi_p}{2}\right)^p$ and the first eigenfunction is represented by $\sin_p(\pi_p x)$, up to a multiplicative constant.
\subsection{Some properties of the eigenvalue problem}
 Now we list some properties of the minimizers of problem \eqref{operat}. 
We argue similarly as in \cite{DP}, where some of these properties have been proved in the case when $p=2$.
\begin{prop}
\label{propr}
Let $\alpha\in \R$, $p\geq2$ and $\frac p2\le r \le p$, then the following properties hold.
\begin{enumerate}
\item[(a)] Problem \eqref{operat} has a solution.
\item[(b)] Any minimizer $y$ of \eqref{operat} satisfies the following boundary value problem
\begin{equation}
\label{el}
\left\{
\begin{array}{ll}
-(|y'|^{p-2}y')' + \alpha \gamma |y|^{r-1}=\lambda_\alpha(p,r)\, |y|^{p-2}y& \text{in}\ ]-1,1[\\[.4cm]
y(-1)=y(1)=0,
\end{array}
\right.
\end{equation}
where 
\[
\gamma=
\begin{cases}
0 &\text{if both } r=p \text{ and }\displaystyle \int_{-1}^1 |y|^{p-1}y\  dx=0, \\
\displaystyle\left|\ds\int_{-1}^1 |y|^{r-1}y\  dx\right|^{\frac{p}{r}-2}\left(\displaystyle \int_{-1}^1 |y|^{r-1}y\  dx \right) &\text{otherwise}.
\end{cases}
\]
Moreover, $y, y'|y'|^{p-2}\in C^1[-1,1]$. 
\item[(c)] The function $\lambda_{\alpha}(p,r)$ is Lipschitz continuous and non-decreasing with respect to $\alpha\in\R$. 
\item[(d)] If $\alpha\le 0$, the minimizers of \eqref{operat} do not change sign in $]-1,1[$, and 
\[
\ds\lim_{\alpha\to -\infty}\lambda_\alpha(p,r)=-\infty.
\]
\item[(e)] We have that
{\[
\ds\lim_{\alpha\to +\infty}\lambda_\alpha(p,r)= \Lambda(p,r)=\pi_p^{p}.
\]}
\end{enumerate}
\end{prop}
\begin{proof}
By the method of Calculus of Variations it is easily proved the existence of a minimizer. Furthermore, any minimizer satisfies \eqref{el}. This follows in a standard way if $r<p$, since the functional $\mathcal Q_\alpha$ in \eqref{Qualpha} is differentiable in $u$. When $r=p$, this functional is not differentiable if $\int_{-1}^1 |y|^{r-1}y\  dx=0.$ Actually, in this case, the problem \eqref{operat} coincides with the minimum of the functional $\mathcal Q_\alpha$ among the functions satisfying $\int_{-1}^1 |y|^{r-1}y\  dx=0$ and, by \cite[Lem. 2.4]{DGS}, it follows that $\gamma=0$. From \eqref{el} immediately follows that $y, y'|y'|^{p-2}\in C^1[-1,1]$ and hence \textit{(a)-(b)} have been proved. 

 In order to get property \textit{(c)}, we stress that for all $\eps>0$, by H\"older inequality, it holds
\begin{equation*}
\mathcal{Q}_{\alpha+\eps} [u]\le\mathcal{Q}_\alpha[u]+\eps\frac{\left(\ds\int_{-1}^{1} |u|^r\  dx\right)^{p/r}}{\ds\int_{-1}^{1} |u|^p\  dx}\leq\mathcal{Q}_\alpha[u]+ 2^\frac{p-r}{r}\eps,\quad\forall\, \eps>0.
\end{equation*}
Therefore the following chain of inequalities
\begin{equation*}
\mathcal{Q}_\alpha[u]\leq\mathcal{Q}_{\alpha+\eps}[u]\le\mathcal{Q}_\alpha[u]+ 2^\frac{p-r}{r}\eps,\quad\forall \ \varepsilon>0,
\end{equation*}
implies, taking the minimum as $u\in W_0^{1,p}(-1,1)$, that
\begin{equation*}
\lambda_\alpha (p,r)\leq\lambda_{\alpha+\varepsilon} (p,r)\leq\lambda_ \alpha (p,r)+2^\frac{p-r}{r}\varepsilon,\quad\forall \ \varepsilon>0,
\end{equation*}
that proves \textit{(c)}. 
If $\alpha< 0$, then
\[
\mathcal{Q}_\alpha[u] \ge \mathcal Q_{\alpha}[|u|],
\]
with equality if and only if $u\ge 0$ or $u\le 0$. Hence any minimizer has constant sign in $]-1,1[$. Finally, it is clear from the definition that $\ds\lim_{\alpha\to-\infty}\lambda_\alpha(p,r)=-\infty$. Indeed, by fixing a positive test function $\varphi$ we get
\[
\lambda_\alpha(p,r) \le \mathcal Q_{\alpha}[\varphi].
\]
Being $\varphi>0$ in $]-1,1[$, then  $\mathcal Q_{\alpha}[\varphi] \to -\infty \quad\text{as }\alpha\to -\infty$,
 and the proof of \textit{(d)} is completed.
The problem \eqref{twist} was studied, for example, in \cite{CD, GN} and the minimum $\Lambda(p,r)$ is equal to $\pi_p^p$.  In particular, if there exists a minimizer $y$ of $\lambda_\alpha(p,r)$ such that $\int_{-1}^{1}|y|^{r-1}y\,dx=0$, then it holds that $\gamma=0$ in \eqref{el}. Indeed, in such a case $y$ is a minimizer also of the problem \eqref{twist}, whose Euler-Lagrange equation is
\begin{equation*}
\left\{
\begin{array}{ll}
-(|y'|^{p-2}y')' =\lambda_\alpha(p,r)\, |y|^{p-2}y\quad &\text{in}\ ]-1,1[,\\[.4cm]
y(-1)=y(1)=0.
\end{array}
\right.
\end{equation*}
Since $\lambda(\alpha, p, r)$ is decreasing with respect to $\alpha$, we have that $\lambda_\alpha(p,r)\le  \Lambda(p,r)=\pi_p^{p}$. Now, let $\alpha_k\ge 0$, $k_n\in\N$, be a positively divergent sequence. For any $k$, we consider a minimizer $y_k\in W_0^{1,p}$ of (\ref{operat}) such that $\|y_k\|_{L^p}=1$. We have that
\begin{equation*}
\lambda_{\alpha_k} (p,r)=\int_{-1}^{1} |y'_k|^p\  dx + \alpha_k\left(\ds\int_{-1}^{1} |y_k|^{r-1}y_k \  dx\right)^\frac{p}{r}\leq \Lambda(p,r).
\end{equation*}
Then $y_k$ converges (up to a subsequence) to a function $y\in W_0^{1,p}(-1,1)$, strongly in $L^p$ and weakly in $W_0^{1,p}$. Moreover $\|y\|_{L^p}=1$ and 
\begin{equation*}
\left( \int_{-1}^{1} |y_k|^{r-1}y_k \  dx\right)^\frac{p}{r}\leq\frac{ \Lambda(p,r)}{\alpha_k}\rightarrow 0 \quad\text{as}\ k\rightarrow + \infty
\end{equation*}
which gives that $\int_{-1}^1 |y|^{r-1}y \  dx=0$. On the other hand the weak convergence in $W_0^{1,p}$ implies that
\begin{equation}
\label{lsclap}
\int_{-1}^{1} |y'|^p\  dx \leq \liminf_{k \rightarrow \infty}\int_{-1}^{1} |y'_k|^p\  dx.
\end{equation}
Therefore, by the definitions of $\Lambda(p,r)$ and $\lambda_\alpha(p,r)$, and by (\ref{lsclap}) we have
\begin{equation*}
\begin{split}
 \Lambda(p,r)\le
\int_{-1}^{1} |y'|^p\  dx &\leq \liminf_{k \rightarrow \infty}\left[\int_{-1}^{1} |y'_k|^p\  dx
+ \alpha_k\left( \int_{-1}^{1} |y_k|^{r-1}y_k \  dx\right)^\frac{p}{r}\right]\\
&\leq\lim_{k \rightarrow \infty}\lambda(\alpha_k,p,r)\leq \Lambda(p,r).
\end{split}
\end{equation*}
and the property \textit{(e)} follows.
\end{proof}

\begin{rem} Let us observe that when $\lambda_\alpha(p,r)=0$, we have (as in \cite{DP}):
\begin{equation*}
\label{lapuq2}
-\alpha=\min_{w\in W_0^{1,p}(-1,1)}\frac{\ds\int_{-1}^{1} |w'|^p dx}{\left(\ds\int_{-1}^{1} |w|^{r} \ {d}x\right)^{p/r}}.
\end{equation*}
\end{rem}

\section{The symmetry of the solutions}\label{section_prop_min}
The main result of this Section, contained in Proposition \ref{cambiosegno}, consists in the fact that each minimizer of problem \eqref{operat} is represented by a generalized sine function, that is symmetric and whose $(r-1)$-power has zero average. This result will allow us to prove, in the following Section, the existence of a critical value of the parameter for the problem \eqref{operat} such that the minimizers are symmetric above this value. 

A key role in the proof of the main results is played by the minimizers that change sign in $]-1,1[$. In the following Proposition we find an expression of the first nonlocal eigenvalue $\lambda_\alpha(p,r)$ with an auxiliary function $H$, whose study leads us to show important properties of problem \eqref{operat}.
\begin{prop}
\label{cambiosegno0}
Let $p\geq2$, $\frac p2\le r \le p$ and suppose that there exists $\alpha>0$ such that $\lambda_\alpha(p,r)$ admits a minimizer $y$ that changes sign in $[-1,1]$. Then the following properties hold.
\begin{enumerate}
\item[(a)] The minimizer $y$ has in $]-1,1[$ exactly one maximum point, $\eta_{M}$, and exactly one minimum point, $\eta_{\bar m}$, and, up to a multiplicative constant,  is such that $y(\eta_{M})=1$ and $y(\eta_{\bar m})= - \bar m\in\ ]0,1]$. 
\item[(b)] If $y_{+}\ge0$ and $y_{-}\le 0$ are, respectively, the positive and negative part of $y$, then $y_{+}$ and $y_{-}$ are, respectively, symmetric about $x=\eta_{M}$ and $x=\eta_{\bar m}$.
\item[(c)] There exists a unique zero of $y$ in $]-1,1[$.
\item[(d)] In the minimum value $\bar m$ of $y$, it holds that  
\[
\lambda_\alpha(p,r)=(p-1)H(\bar m,p,r)^p,
\] 
where $H(m,p,r)$, $(m,p,r)\in[0,1]\times[2,+\infty[\times\left[\frac p2,p\right]$, is the function defined as
\begin{equation*}
\begin{split}
&H(m,p,r):= \int_{-m}^1\frac{dy}{ [1- R(m,p,r)(1- |y|^{r-1}y) - |y|^p]^\frac1p}=\\[.3cm]
=\int_{0}^1 & \frac{dy}{ [1-R(m,p,r)(1-y^{r}) - y^p]^\frac1p} + \int_{0}^1\frac{mdy}{ [1-R(m,p,r)(1+m^r y^{r}) -m^p y^p]^\frac1p}
\end{split}
\end{equation*}
and $R(m,p,r)=\frac{1-m^p}{1+m^r}$.
\end{enumerate}
\end{prop}

\begin{proof}
Let us suppose that $\lambda_\alpha(p,r)$ admits a minimizer $y$ that changes sign and that
\[
\max_{[-1,1]} y(x)=1,\quad \min_{[-1,1]} y(x)=-\bar m,\quad \text{with}\ \bar m\in]0,1].
\]
It is always possible to reduce to this condition by multiplying the solution for a suitable positive constant. 
Let us consider $\eta_{M},\eta_{\bar m}$ in $]-1,1[$ such that $y(\eta_{M})=1=\max_{[-1,1]}y$, and  $y(\eta_{\bar m})=-\bar m=\min_{[-1,1]}y$. For the sake of simplicity, we will write $\lambda=\lambda_\alpha(p,r)$.
If we multiply the equation in \eqref{el} by $y'$ and integrate, we get
\begin{equation}
\label{inel1d}
\frac{|y'|^p}{p'}+\lambda \frac{|y|^p}{p} = \frac{\alpha\gamma}{r} |y|^{r-1}y + c\qquad\ \text{in } ] -1,1[,
\end{equation}
for a suitable constant $c$ and $\frac1p + \frac{1}{p'}=1$. 
Being $y'(\eta_{M})=0$ and $y(\eta_{M})=1$, we have
\begin{equation}
\label{cosmax}
c=\frac{\lambda}{p}-\frac{\alpha}{r} \gamma.
\end{equation}
Moreover, $y'(\eta_{\bar m})=0$ and $y(\eta_{\bar m})=-m$ give also that
\begin{equation}
\label{cosmin}
c=\lambda\frac{\bar m^p}{p}+\frac{\alpha}{r} \bar m^r\gamma.
\end{equation}
Joining \eqref{cosmax} and \eqref{cosmin}, we obtain
\begin{equation}
\label{costnl}
\left\{
\begin{array}{l}
 \gamma =\frac{\displaystyle r\lambda}{\displaystyle p \alpha}R(\bar m,p,r)\\[.3cm]
c=\frac{\ds\lambda}{\displaystyle p}T(\bar m,p,r)
\end{array}
\right.
\end{equation}
where
\begin{equation*}
R(m,p,r)=\frac{1-m^p}{1+m^r} \quad\text{and}\quad T(m,p,r)=\frac{m^p+m^r}{1+m^r}=1-R(m,p,r).
\end{equation*}
Then \eqref{inel1d} can be written as 
\begin{equation}\label{integratedELconstant}\frac{|y'|^p}{p'}+\lambda \frac{|y|^p}{p} = \frac{\lambda}{p} R(\bar m,p,r) |y|^{r-1}y + \frac{\lambda}{p} (1-R(\bar m,p,r))\qquad\ \text{in }  ]-1,1[.\end{equation}
From \eqref{integratedELconstant}, we have
\begin{equation*}
\label{integratedel}
|y'|^p=\frac{\lambda}{p-1} (1-R(\bar m,p,r)(1- |y|^{r-1}y) - |y|^p) \qquad\ \text{in}  ]-1,1[.
\end{equation*}
It is easy to see that the number of zeros of $y$ has to be finite, hence let 
\[
-1=\zeta_{1}<\ldots<\zeta_{j}<\zeta_{j+1}<\ldots<\zeta_{n}=1
\] 
be the zeroes of $y$. As observed in \cite{CD}, it is easy to show that 
\begin{equation*}
\label{dac}
y'(x)=0 \iff y(x)=-\bar m\text{ or }y(x)=1.
\end{equation*}
This implies that $y$ has no other local minima or maxima in $]-1,1[$, and in any interval $]\zeta_{j},\zeta_{j+1}[$ where $y>0$ there is a unique maximum point, and in any interval $]\zeta_{j},\zeta_{j+1}[$ where $y<0$ there is a unique minimum point. 

Now, we set
\[
g(Y):=1-R(\bar m,p,r)(1-|Y|^{r-1}Y)-Y^{p},\quad Y\in [-\bar m,1],
\]
and we have
\begin{equation}
\label{CDproof}
|y'|^{p}=\frac{\lambda}{p-1}\, g(y).
\end{equation}
Let us observe that $g(-\bar m)=g(1)=0$. Being $p\ge r$, it holds that $g'(\bar Y)=0$ implies $g(\bar Y)>0$. Hence, $g$ does not vanish in $]-\bar m,1[$. By \eqref{CDproof}, it holds that $y'(x)\ne 0$ if $y(x)\ne 1 $ and $y(x)\ne -\bar m$. 

Now, we will adapt the argument of \cite[Lemma 2.6]{DGS}. The following three claims below allow to complete the proof of \textit{(a)}, \textit{(b)} and \textit{(c)}.

\begin{description}
\item[Claim 1:] in any interval $]\zeta_{j},\zeta_{j+1}[$ given by two subsequent zeros of $y$ and in which $y=y^{+}>0$, has the same length; in any of such intervals, $y^{+}$ is symmetric about $x=\frac{\zeta_{j}+\zeta_{j+1}}{2}$;
\item[Claim 2:] in any interval $]\zeta_{j},\zeta_{j+1}[$ given by two subsequent zeros of $y$ and in which $y=y^{-}<0$ has the same length; in any of such intervals, $y^{-}$ is symmetric about $x=\frac{\zeta_{j}+\zeta_{j+1}}{2}$;
\item[Claim 3:] there is a unique zero of $y$ in $]-1,1[$.
\end{description}

This result was proved in the case $p=2$ in \cite{DP} and following this proof, we can show the result in the hypothesis of the Proposition. Properties \textit{(a)}, \textit{(b)} and \textit{(c)} can be also proved by using a symmetrization argument, by rearranging the functions $y^{+}$ and $y^{-}$ and using the P\'olya-Szeg\H o inequality and the properties of rearrangements (see also, for example, \cite{BFNT} and \cite{D}).

Now denote by $\eta_{M}$ and $\eta_{\bar m}$, respectively, the unique maximum and minimum point of $y$. It is not restrictive to suppose $\eta_{M}<\eta_{\bar m}$. They are such that $\eta_{M}-\eta_{\bar m}=1$, with $y'<0$ in $]\eta_{M},\eta_{\bar m}[$. Then
\begin{equation*}{\lambda}^\frac 1 p=(p-1)^\frac 1p\frac{-y'}{[1-R(\bar m,p,r)(1-|y|^{r-1}y)- y^p]^\frac 1 p} \qquad\ \text{in} \ ]\eta_{M},\eta_{\bar m}[.
\end{equation*}
Integrating between $\eta_{M}$ and $\eta_{\bar m}$, we have
\begin{equation*}
\lambda=(p-1)\left[\int_{-\bar m}^1\frac{dy}{ [1-R(\bar m,p,r)(1-|y|^{r-1}y)- y^p]^\frac 1 p}\right]^{p}= (p-1)H(\bar m,p,r)^p,
\end{equation*}
and the proof of the Proposition is completed.
\end{proof}

To prove the main result of this Section, we will show the monotonicity of the function $H(m,p,r)$, defined in Proposition \ref{cambiosegno0}, with respect to $r$ (Lemma \ref{prop-monotonia-r}) and with respect to $m$ (Lemma \ref{Hdecr}).

The proof of the monotonicity with respect to $r$ is based on the study of the integrand function that defines $H(m,p,r)$, that is
\begin{equation*}
h(m,p,r,y):=\frac{1}{ [1-R(m,p,r)(1-y^{r}) - y^p ]^\frac 1 p}
+\frac{m}{ [1-R(m,p,r)(1+ m^{r}y^{r})- m^py^p ]^\frac 1p},
\end{equation*}
for $y\in[0,1]$.
Let us explicitly observe that if $m=1$, then $z(1,p,r)=0$ and
\[
h(1,p,r,y)=\frac{2}{ [1- y^p]^\frac 1 p}, 
\]
that is constant in $r$. Moreover, if $y=0$, then
\[
h(m,p,r,0)=\frac{1+m}{[1-R(m,p,r) ]^\frac 1 p}
\]
that is strictly increasing in $r\in[\frac p2,p]$. 

\begin{lem}
\label{prop-monotonia-r}
For any fixed $y\in[0,1[$ and $m\in]0,1[$, the function $h(m,p,\cdot,y)$ is strictly increasing with respect to $r$ as $\frac p2\le r\le p$.
\end{lem}
\begin{proof}
From the preceding observations,  we may assume $m\in]0,1[$ and $y\in]0,1[$.
Differentiating in $r$, we have, for $R=R(m,p,r)$, that
\begin{equation*}
\begin{split}
\de_{r}h = &- \frac {1}{pF_{I}^{p+1}} \big[-(1-y^{r})\de_{r} R + R\, y^{r}\log y \big]+\\[.2cm]
&- \frac {m}{pF_{I\!I}^{p+1}} \big[
-(1+m^{r}y^{r})\de_{r}R- R\,m^{r}y^{r}(\log m +\log y)
\big],
\end{split}
\end{equation*}
where
\begin{equation}
\label{firint}
F_{I}(m,p,r,y):=\displaystyle \left[1-R(1- y^{r}) - y^p\right]^\frac 1 p \le \left[1-y^p\right]^\frac 1p,
\end{equation}
and 
\begin{equation}
\label{secint}
F_{I\!I}(m,p,r,y):=\displaystyle \left[1-R(1+m^{r}y^{r})-m^{p}y^{p}\right]^\frac 1 p\ge m\left[1-y^{p}\right]^\frac 1p.
\end{equation}
Being 
\[
R=\frac{1-m^{p}}{1+m^{r}},\quad
\de_{r}R= -\frac{1-m^{p}}{(1+m^{r})^{2}}m^{r}\log m,
\]
we have that
\begin{equation*}
\begin{split}
\de_{r}h = \frac1p\frac{1-m^{p}}{(1+m^{r})^{2}}\bigg\{ &\overbrace{\bigg[ -(1-y^{r})m^{r}\log m -y^{r}(1+m^{r})\log y\bigg]}^{h_{1}(m,r,y)}\frac{1}{F_{I}^{p}}+ \\[.2cm]
+& \underbrace{\bigg[-(1+m^{r}y^{r})\log m+(1+m^{r})y^{r}(\log m +\log y)\bigg]}_{h_{2}(m,r,y)}\frac {m^{r+1}}{F_{I\!I}^{p}}\bigg\}.
\end{split}
\end{equation*}

Let us observe that $h_{1}(m,p,r,y)\ge 0$. Hence, in the set $A$ of $(m,p,r,y)$ such that $h_{2}(m,p,r,y)$ is nonnegative, we have that $\de_{r} h(m,p,r,y) \ge 0$. Moreover, $h_{1}(m,p,r,y)$ cannot vanish ($y<1$), then $\de_{r}h>0$ in $A$.

Hence, let us consider the set $B$ where
\[
h_{2}=(y^{r}-1)\log m+(1+m^{r})y^{r}\log y\le 0
\] 
(observe that in general $A$ and $B$ are nonempty). 
By \eqref{firint} and \eqref{secint} we have that
\begin{equation*}
\begin{split}
\de_{r}h \ge \frac1p\frac{1-m^{p}}{(1+m^{r})^{2}}\bigg\{ &\bigg[ -(1-y^{r})m^{r}\log m -y^{r}(1+m^{r})\log y\bigg]\frac{1}{(1-y^{p})^{\frac {p+1}{p}}}+ \\[.2cm]
+& \bigg[(y^{r}-1)\log m+(1+m^{r})y^{r}\log y\bigg]\frac {m^{r-p}}{(1-y^{p})^{\frac{p+1}{p}}}\bigg\}.
\end{split}
\end{equation*}
Hence, to show that $\de_{r}h> 0$ also in the set $B$ it is sufficient to prove that
\begin{multline}
\label{g>0}
g(p;m,r,y):=\bigg[ -(1-y^{r})m^{r}\log m -y^{r}(1+m^{r})\log y\bigg]+\\
+ \bigg[(y^{r}-1)\log m+(1+m^{r})y^{r}\log y\bigg]m^{r-p}> 0
\end{multline}
when $m\in]0,1[$, $r\in[\frac p2,p]$ and $y\in]0,1[$.\\

{\bf Claim 1.} {\em For any $r\in \left[\frac p2,p\right]$ and $m\in]0,1[$, the function $g(m,r,\cdot)$ is strictly decreasing for $y\in]0,1[$.} 
\\

To prove the Claim 1, we differentiate $g$ with respect to $y$, obtaining
\begin{multline*}
\de_{y}g= \bigg[ry^{r-1}m^{r}\log m-ry^{r-1}(1+m^{r})\log y-y^{r-1}(1+m^{r})\bigg]+\\
+\bigg[ry^{r-1}\log m+(1+m^{r})(ry^{r-1}\log y+y^{r-1})
\bigg]m^{r-p}=\\
=y^{r-1}\bigg[r (m^{r} +m^{r-p})\log m+r(1+m^{r})(m^{r-p}-1)\log y + (1+m^{r})(m^{r-p}-1)\bigg].
\end{multline*}
Then $\de_{y} g < 0$ if and only if
\begin{equation*}
(1+m^{r})(m^{r-p}-1)(r\log y+1) < -r(m^{r}+m^{r-p})\log m.
\end{equation*}
The above inequality is true, as we will show that (recall that $0<m<1$ and $\frac p2\le r\le p$)
\begin{equation}
\label{ineqlog}
\log y < -\frac1r +\frac{(m^{r}+m^{r-p})\log m}{(1+m^{r})(1-m^{r-p})}=:-\frac1r+\ell(m,r). 
\end{equation}
If the the right-hand side of \eqref{ineqlog} is nonnegative, then for any $y\in ]0,1[$ the inequality \eqref{ineqlog} holds.\\

{\bf Claim 2.} {\em For any $r\in \left[\frac p2,p\right]$ and $m\in]0,1[$, $\ell(m,r)> \frac1r$.} \\

We will show that
\[
\ell(m,r) >  \frac 1r.
\]
We have
\[
\ell(m,r)= \frac{(m^{r}+m^{r-p})\log m }{(1+m^{r})(1-m^{r-p})}> \frac 1r
\]
if and only if
\begin{multline*}
\mu(m,r)=(m^{r}+m^{r-p})\log\frac{1}{m} - \frac1r(1+m^{r})(m^{r-p}-1)=\\
=(m^{r}+m^{r-p})\log\frac{1}{m} +\frac1r ( 1+m^{r}-m^{r-p}-m^{2r-p})=\\
= m^{r}\left(\log\frac 1m+\frac1r \right)+ m^{r-p}\left(\log\frac1m-\frac1r \right)+\frac1r (1-m^{2r-p})
 > 0.
\end{multline*}
Then for $m \in]0,1[$ we have
\begin{multline*}
\mu(m,r)=m^{r}\left(\log\frac 1m+\frac1r \right)+ m^{r-p}\left(\log\frac1m-\frac1r\right)+\frac1r (1-m^{2r-p})\\
\ge m^{r}\left(\log\frac 1m+\frac1r\right)+ m^{r-p}\left(\log\frac1m-\frac1r\right)=\\=
m^{r-p}\left(m^{p}\left(\log\frac 1m+\frac1r\right)+\log\frac 1m-\frac1r\right):=m^{r-p}\eta (m,r)> 0.
\end{multline*}
We prove that $\mu (m,r)$ is positive by showing that $\eta (m,r)$ is decreasing in $m$:
\[
\partial_m \eta (m,q)=m^{p-1}\left(\log \frac 1{m^p} - \frac1{m^p}+\frac pr -1\right).
\]
Since $\log \frac 1{m^p} < \frac1{m^p}-1$, we have that $\partial_m \eta (m,q)<0$ when $r\geq \frac p2$ and the Claim 2, and then the Claim 1, are proved. To conclude the proof of \eqref{g>0}, it is sufficient to observe that 
\[
g(m,r,y)> g(m,r,1)=0
\]
when $m\in]0,1[$, $r\in[\frac p2,p]$ and $y\in]0,1[$.

The Claim 1 gives that $\de_{r}h(m,r,y)> 0$ when $m\in]0,1[$, $r\in\left[\frac p2,p\right]$ and $y\in]0,1[$, and this conclude the proof.
%
%
%
\end{proof}
Now, to prove the monotonicity of $H$ in $m$, we argue similarly as in \cite{GGR}. We show that, for any fixed $p\geq 2$ the function $K(m):=H\left(m,p,\frac p2\right)$ is constant.
\begin{lem}
\label{Hdecr} Let $p\geq2$, then $K'(m)= 0$, $\forall\ m\in]0,1[$.
\end{lem}
\begin{proof} For any fixed $p\geq2$, we denote the following non negative function by:
\begin{align*}
&A(m,y):= m^\frac p 2 +(1-m^\frac p2)y^{\frac p2}-y^p,\quad\forall\ {(m,y)\in[0,1]^2};\\
&B(m,y):= m^\frac p 2-(1-m^\frac p2)m^\frac p 2 y^\frac p2-m^p y^p,\quad\forall \ {(m,y)\in[0,1]^2}.
\end{align*}
Moreover, in this case
\[
R\left(m,p,\frac p2\right)=1-m^\frac p2,\quad\forall \ {m\in[0,1]}.
\]
Hence $K(m)=\ds\int_0^1 \left(A(m,y)^{-\frac 1p}+mB(m,y)^{-\frac 1p}\right)dy$ and 
\begin{multline*}
K'(m)=-\frac 1 p\int_0^1 \left(A(m,y)^{-\frac 1p-1}\frac{\partial A(m,y)}{\partial m}\right.+\\
\left. B(m,y)^{-\frac 1p-1}\left(-p B(m,y)+m\frac{\partial B(m,y)}{\partial m}\right)\right)dy.
\end{multline*}
Differentiating with respect to $m$, we obtain
\begin{align*}
&\frac{\partial A(m,y)}{\partial m}=\frac p 2 m^{\frac p2 -1}(1-y^{\frac p2}),\\
&-pB(m,y)+m\frac{\partial B(m,y)}{\partial m}=-\frac p 2 m^{\frac p2}(1-y^{\frac p2}).
\end{align*}
Hence
\[
K'(m)=\frac{m^\frac p2-1}{2}\int_0^1 \left(-\frac{1-y^{\frac p2}}{A(m,y)^{\frac 1p+1}} +\frac{m(1-y^{\frac p2})}{B(m,y)^{\frac 1 p +1}}\right)dy.
\]
Now we study the sign of the right integral. We want to prove that
\begin{equation}
\label{Gobbino_int}
\int_0^1 \frac{1-y^{\frac p2}}{A(m,y)^\frac{p+1}{p}}\ dy = \int_0^1 \frac{m (1-y^{\frac p2})}{B(m,y)^\frac{p +1}{p}} \ dy
\end{equation}
Following the ideas of \cite{GGR}, for all $m\in (0,1)$, we set
\begin{equation*}
\label{delta}
\delta (y):=[1-(1-m^{\frac p2})y^{\frac p2}]^\frac{2}{p}\quad\forall\in [0,1]
\end{equation*}
and
\begin{equation}
\label{changingvariables}
h(y):=\frac{my}{\delta (y)}\quad\forall\ y\in[0,1].
\end{equation}
It holds that $h(0)=0$, $h(1)=1$ and
\[
h'(y):=\frac{m}{\delta (y)^{\frac p2+1}}, \ \forall\ y \in (0,1).
\]
Hence the function $h$ is strictly increasing and, keeping \eqref{changingvariables} into account, the result follows if we prove that
\begin{multline*}
\int_{0}^1{ \frac{ 1-m^{\frac p2}y^{\frac p2}\delta (y)^{-\frac p2}}{\displaystyle \left(m^\frac p2+(1-m^\frac p2)m^{\frac p2}y^{\frac p2}{\delta (y)^{-\frac p2}}- \displaystyle m^{p}y^p{\delta (y)^{-p}}\right)^\frac{p+1}{p}}\cdot\frac{1}{\displaystyle \delta(y)^{\frac p2+2}} \ dy}\\
= \int_{0}^1\frac{ 1-y^{\frac p2}}{\displaystyle \left(m^\frac p2-(1-m^\frac p2)m^{\frac p2}y^{\frac p2}-m^py^p\right)^\frac{p +1}{p}} \ dy
\end{multline*}
Therefore \eqref{Gobbino_int} is proved if we show that
\begin{multline*}
\displaystyle  m^\frac p2-(1-m^\frac p2)m^{\frac p2}y^{\frac p2}-m^py^p   = \delta(y)^{p} \left(m^\frac p2+(1-m^\frac p2)\frac{m^{\frac p2}y^{\frac p2}}{\delta (y)^{\frac p2}}- \frac{\displaystyle m^{p}y^p}{\delta (y)^p}\right) ,
\end{multline*}
and this is an equality that can be easily checked.
\end{proof}

Now, we are in position to state the main property of the function $H(m,p,r)$.
\begin{lem}
\label{propminimo} Let $p\geq2$ and $\frac p2\le r\le p$, then for all $m\in[0,1]$  it holds that
\begin{equation*}
H(m,p,r) \ge\frac{\pi_p}{(p-1)^\frac 1p}.
\end{equation*}
Moreover:
\begin{itemize}
\item when $\frac p2 < r\le p$, then $H(m,p,r)=\frac{\pi_p}{(p-1)^\frac 1p}$ if and only if $m=1$; 
\item $H\left(m,p,\frac p2\right)=\frac{\pi_p}{(p-1)^\frac 1p}$ for all $m\in[0,1]$.
\end{itemize}
\end{lem}

\begin{proof}
If $m=1$, we have that
\[
H(1,p,r)=2\int_0^1\frac{dy}{[1-y^p]^\frac1p}=\frac{\pi_p}{(p-1)^\frac 1p}
\]
for  $1\le\frac p2\le r\le p$. Moreover, by Lemma \ref{Hdecr}
\[
H\left(m,p,\frac p2\right)= H\left(1,p,\frac p2\right)=\frac{\pi_p}{(p-1)^\frac 1 p}.
\]
for any $m\in[0,1]$.

To study all the other cases, we first consider $0<m<1$. Then for any $p\ge 2$, $\frac p2 < r\le p$, by Lemma  \ref{prop-monotonia-r}  we get
we have
\[
H(m,p,r)> H\left(m,p,\frac p2\right) = \frac{\pi_p}{(p-1)^\frac 1 p}.
\]
When $m=0$, simple calculations give
\[
H(0,p,r)=\int_0^1\frac{dy}{[y^r-y^p]^\frac1p}>\int_0^1\frac{dy}{[y^{\frac p2}-y^p]^\frac1p}=H\left(0,p,\frac p2\right)= \frac{\pi_p}{(p-1)^\frac 1p}
\]
and hence the result.
\end{proof}
\begin{prop}\label{cambiosegno} Let $p\geq2$, $\frac p2\le r \le p$ and suppose that there exists $\alpha>0$ such that $\lambda_\alpha(p,r)$ admits a minimizer $y$ that changes sign in $[-1,1]$.
\item[(i)] If $\frac p 2\le r \le p$, then
\[
\lambda_\alpha(p,r)=\Lambda(p,r)=\pi_p^{p}.
\] 
\item[(ii)] If $\frac p 2< r \le p$, then
\begin{equation}
\label{r-average}
\int_{-1}^{1}|y|^{r-1}y\,dx=0.
\end{equation}
\item[(iii)] If $\frac p 2\le r \le p$ and \eqref{r-average} holds, then {$y(x)=C\sin_p (\pi_p x)$, with $C\in \R\setminus\{0\}$}. Hence the only point in $]-1,1[$ where $y$ vanishes is $\overline x=0$. 
\end{prop}

\begin{proof}
Let us consider a minimizer $y$ of $\lambda_\alpha(p,r)$ in $[-1,1]$ that changes sign, with $\max y=1$ and $\bar m=-\min y$.

By  \textit{(d)} of Proposition \ref{cambiosegno0} and Lemma \ref{propminimo}, the eigenvalue $\lambda_\alpha(p,r)$ has to satisfy the inequality
\[
\lambda_\alpha(p,r)\ge \pi_p^{p}.
\]
Hence, by \textit{(c)} and \textit{(e)} of Proposition \ref{propr}, it follows that
\[
\lambda_\alpha(p,r) = \pi_p^{p},
\]
that gives $\textit{(i)}$.

Now assume that $\frac p 2<r\le p$. Again by Lemma \ref{propminimo} and \textit{(d)} of Proposition \ref{cambiosegno0}, $\lambda_\alpha(p,r)=\pi_p^{p}$ if and only if $\bar m=1$. Hence, the first identity of \eqref{costnl} gives that 
\[
\int_{-1}^{1} |y|^{r-1}ydx=0,
\]
and \textit{(ii)} follows. To prove \textit{(iii)}, let us explicitly observe that, when \eqref{r-average} holds, $y$ solves
\[
\begin{cases}
(|y'|^{p-2}y')'+\pi^{p}_p|y|^{p-2}y=0&\text{in }]-1,1[\\
y(-1)=y(1)=0.&\\
\end{cases}
\]
Hence {$y(x)=C\sin_p (\pi_p x)$, with $C\in \R\setminus\{0\}$}.
%
\end{proof}

\section{Proof of the main results}
In this Section we prove the main results by using the properties of Section \ref{section_prop_min}.
\begin{proof}[Proof of Theorem \ref{mainthm1}]
We prove that there exists a critical value of the parameter such that the minimizer is symmetric. Firstly we prove the following claim.\\

\textbf{Claim.}
\textit{There exists a positive value of $\alpha$ such that the minimum problem
\begin{equation*}
\lambda_\alpha(p,r)=\min_{u \in W_0^{1,p}([-1,1])}\frac{\displaystyle \int_{-1}^{1} | u'|^p\  dx+\alpha \left|\int_{-1}^{1} u|u|^{r-1}\  dx\right|^{\frac{p}{r}}}{\ds\int_{-1}^{1} u^p\  dx}
\end{equation*}
admits an eigenfunction $y$ that satisfies $\int_{-1}^1y|y|^{r-1}\  dx=0$. In such a case, $\lambda_\alpha(p,r)=\pi_p^{p}$ and, up to a multiplicative constant, $y=\sin_p (\pi_p x)$.}

By Proposition \ref{propr} \textit{(e)}, if a minimizer $y$ changes sign, then we may suppose that $\alpha>0$. By contradiction, we suppose that for any $k\in\N$, there exists a divergent sequence $\alpha_{k}$, and a corresponding sequence of eigenfunctions $\{y_{k}\}_{k\in\N}$ relative to $\lambda_{\alpha_{k}}(p,r)$ such that $\int_{-1}^1y_{k}|y_{k}|^{r-1}dx>0$ and $\|y_k\|_{L^p(-1,1)}=1$. By Proposition \ref{cambiosegno}, these eigenfunctions do not change sign and, as we have already observed, $\lambda_{\alpha_{k}}(p,r)\le\pi_p^{p}$. Hence, it holds that
\begin{equation}
\label{contrad}
\displaystyle \int_{-1}^{1} | y_k'|^p\  dx+\alpha_k \left(\int_{-1}^{1} |y_k|^{r}\  dx\right)^{\frac{p}{r}}\le\pi_p^{p}.
\end{equation}
Therefore, $y_k$ converges (up to a subsequence) to a function $y\in W_0^{1,p}(-1,1)$, strongly in $L^p(-1,1)$ and weakly in $W_0^{1,p}(-1,1)$. Moreover $\|y\|_{L^p(-1,1)}=1$ and $y$ is not identically zero. Therefore $\|y\|_{L^r(-1,1)}>0$ and, letting $\alpha_k\rightarrow +\infty$ in \eqref{contrad} we have a contradiction and the claim is proved.

Now, let us recall that, by Proposition \ref{propr}, $\lambda_\alpha(p,r)$ is a nondecreasing Lipschitz function in $\alpha$.  Therefore we can define 
\[
\alpha_{C}=\min\{\alpha\in\R\colon \lambda_\alpha(p,r)=\pi^p_{p}\}=\sup\{\alpha\in\R\colon \lambda_\alpha(p,r)<\pi_p^{p}\}.
\]
We easily verify that this value of the parameter is positive and if $\alpha<\alpha_{C}$, then the minimizers corresponding to $\lambda_\alpha(p,r)$ have constant sign, otherwise $\lambda_\alpha(p,r)=\pi_p^{p}$. When $\alpha>\alpha_{C}$, then any minimizer $y$ corresponding to $\alpha$ is such that $\int_{-1}^{1}|y|^{r-1}y\,dx=0$. Indeed, if we assume, by contradiction, that there exist $\bar \alpha>\alpha_{C}$ and $\bar y$ such that $\int_{-1}^{1}|\bar y|^{r-1}\bar y\,dx>0$, $\|y\|_{L^p}=1$ and $\mathcal Q_{\bar\alpha}[\bar y]=\lambda_{\bar\alpha}(p,r)$, then 
\begin{align*}
\mathcal Q_{\bar\alpha-\eps}[\bar y] &=
\mathcal Q_{\bar\alpha}[\bar y]-\eps\left(\int_{-1}^{1}|\bar y|^{r-1}\bar y\,dx\right)^{\frac pr }\\ &=\lambda_{\bar\alpha}(p,r)-\eps\left(\int_{-1}^{1}|\bar y|^{r-1}\bar y\,dx\right)^{\frac pr}<\lambda_{\bar\alpha}(p,r).
\end{align*}
Hence, for $\eps$ sufficiently small, $\pi_p^{p}=\lambda_{\alpha_{r}}(p,r)\le\lambda_{\bar\alpha-\eps}(p,r)<\lambda_{\bar\alpha}(p,r)$ and this is absurd.
Finally, by \textit{(iii)} of Proposition \ref{cambiosegno}, the proof of Theorem \ref{mainthm1} is completed. 
\end{proof}

\begin{proof}[Proof of Theorem \ref{mainthm2}]
It is not difficult to see, by means of approximating sequences, that $\lambda_{\alpha_{r}}(p,r)$ admits both a nonnegative minimizer and a minimizer with vanishing $r$-average. To conclude the proof of Theorem \ref{mainthm2}, we have to study the behavior of the solutions when $r=p$. When $\alpha=\alpha_C(p,p)$, the corresponding positive minimizer $y$ is a solution of
\begin{equation*}
\left\{
\begin{array}{ll}
(|y'|^{p-2}y')'+\pi_p^p y^{p-1}=\alpha_C(p,p) y^{p-1} &\text{in }]-1,1[\\
y(-1)=y(1)=0.
\end{array}
\right.
\end{equation*}
The positivity of the eigenfunction guarantees that 
\begin{equation*}
\pi_p^p-\alpha_C(p,p)=\lambda_{0}(p,p)=\left(\frac{\pi_p}{2}\right)^p,
\end{equation*}
hence $\alpha_{C}(p,p)=\frac{2^p-1}{2^p}\pi_p^p$.
\end{proof}
\begin{rem}
When $\frac p2\le r<p$, we obtain the following lower bound on $\alpha_C(p,r)$:
	\begin{equation}
	\label{stimar}
	\alpha_{C}(p,r)\ge \frac{2^p-1}{2^{\frac pr+p-1}}\pi^{p}_p.
	\end{equation}
	To get the estimate \eqref{stimar}, we use the monotonicity of $\lambda_{\alpha}(p,r)$ with respect to $\alpha$, and consider the test function $u(x)=\cos_p(\frac{\pi_p}{2} x)$. Hence
	\begin{equation*}
	\pi^{p}_p=\lambda_{\alpha_{C}}(p,r)\le\mathcal{Q}[u,\alpha_{C}]=\left(\frac{\pi_p}{2}\right)^p+\alpha_{C}\left(\int_{-1}^{1}u^{r}dx\right)^{p/r}\le \frac{\pi_p^{p}}{2^p} +\alpha_{C}2^{\frac pr-1}.
	\end{equation*}
\end{rem}

\subsection*{Acknowledgement} 
The authors want to thank Professor Bernd Kawohl for his suggestions during the stay of the second author in K\"oln. \\
This work has been partially supported by GNAMPA of INdAM.

\end{document}